\documentclass[11pt]{amsart}
\usepackage{amsfonts,amssymb,amscd,amsmath,enumerate,verbatim,calc,latexsym,pstcol,pst-plot,pst-3d,}

\def\NZQ{\mathbb}               
\def\NN{{\NZQ N}}

\def\ZZ{{\NZQ Z}}

\def\frk{\mathfrak}               

\def\mm{{\frk m}}

\def\Phi{{\frk N}}
\def\zb{{\bold z}}

\def\opn#1#2{\def#1{\operatorname{#2}}} 
\opn\chara{char} \opn\length{\ell} \opn\pd{pd} \opn\rk{rk}
\opn\projdim{proj\,dim} \opn\injdim{inj\,dim} \opn\rank{rank}
\opn\depth{depth} \opn\grade{grade} \opn\height{height}
\opn\embdim{emb\,dim} \opn\codim{codim}

\opn\Tr{Tr} \opn\bigrank{big\,rank}
\opn\superheight{superheight}\opn\lcm{lcm}
\opn\trdeg{tr\,deg}
\opn\reg{reg} \opn\lreg{lreg} \opn\ini{in} \opn\lpd{lpd}
\opn\size{size}\opn{\mult}{mult}
\opn\div{div} \opn\Div{Div} \opn\cl{cl} \opn\Cl{Cl}
\opn\Spec{Spec} \opn\Supp{Supp} \opn\supp{supp} \opn\Sing{Sing}
\opn\Ass{Ass} \opn\Min{Min}
\opn\Ann{Ann} \opn\Rad{Rad} \opn\Soc{Soc}
\opn\Syz{Syz} \opn\Im{Im} \opn\Ker{Ker} \opn\Coker{Coker}
\opn\Am{Am} \opn\Hom{Hom} \opn\Tor{Tor} \opn\Ext{Ext}
\opn\End{End} \opn\Aut{Aut} \opn\id{id} \opn\ini{in}

\opn\nat{nat}
\opn\pff{pf}
\opn\Pf{Pf} \opn\GL{GL} \opn\SL{SL} \opn\mod{mod} \opn\ord{ord}
\opn\Gin{Gin}
\opn\Hilb{Hilb}\opn\adeg{adeg}\opn\std{std}\opn\ip{infpt}
\opn\Pol{Pol}
\opn\sat{sat}
\opn\Var{Var}
\opn\Gen{Gen}

\opn\aff{aff} \opn\con{conv} \opn\relint{relint} \opn\st{st}
\opn\lk{lk} \opn\cn{cn} \opn\core{core} \opn\vol{vol}
\opn\link{link} \opn\star{star}
\opn\gr{gr}

\def\pot#1#2{#1[\kern-0.28ex[#2]\kern-0.28ex]}

\opn\dirlim{\underrightarrow{\lim}}
\opn\inivlim{\underleftarrow{\lim}}
\let\union=\cup
\let\sect=\cap

\let\tensor=\otimes
\let\iso=\cong
\let\Union=\bigcup

\let\Dirsum=\bigoplus

\let\to=\rightarrow
\let\To=\longrightarrow
\def\Implies{\ifmmode\Longrightarrow \else
        \unskip${}\Longrightarrow{}$\ignorespaces\fi}
\def\implies{\ifmmode\Rightarrow \else
        \unskip${}\Rightarrow{}$\ignorespaces\fi}
\def\iff{\ifmmode\Longleftrightarrow \else
        \unskip${}\Longleftrightarrow{}$\ignorespaces\fi}

\let\:=\colon
\newtheorem{Theorem}{Theorem}[section]
\newtheorem{Lemma}[Theorem]{Lemma}
\newtheorem{Corollary}[Theorem]{Corollary}
\newtheorem{Proposition}[Theorem]{Proposition}

\newtheorem{Example}[Theorem]{Example}

\newtheorem{Conjecture}[Theorem]{Conjecture}

\let\epsilon\varepsilon
\let\phi=\varphi
\let\kappa=\varkappa
\def\qed{\ifhmode\textqed\fi
      \ifmmode\ifinner\quad\qedsymbol\else\dispqed\fi\fi}
\def\textqed{\unskip\nobreak\penalty50
       \hskip2em\hbox{}\nobreak\hfil\qedsymbol
       \parfillskip=0pt \finalhyphendemerits=0}
\def\dispqed{\rlap{\qquad\qedsymbol}}

\opn\dis{dis}
\def\pnt{{\raise0.5mm\hbox{\large\bf.}}}

\opn\Lex{Lex}

\begin{document}

\title{Powers of componentwise linear ideals}

\author{ J\"urgen Herzog, Takayuki Hibi and Hidefumi Ohsugi}
\thanks{This paper was written while the  first author was staying at Osaka University}

\address{J\"urgen Herzog, Fachbereich Mathematik, Universit\"at Duisburg-Essen, Campus Essen, 45117
Essen, Germany} \email{juergen.herzog@uni-essen.de}

\address{Takayuki Hibi, Department of Mathematics, Graduate School of Science,
Osaka University, Toyonaka, Osaka 560-0043, Japan}
\email{hibi@math.sci.osaka-u.ac.jp}

\address{Hidefumi Ohsugi, Department of Mathematics, College of Science, Rikkyo University,
Tokyo 171-8501, JAPAN}
\email{ohsugi@rkmath.rikkyo.ac.jp}
\dedicatory{Dedicated to the memory of Professor Masayoshi Nagata}

\begin{abstract}
We give criteria for graded ideals to have the property that all their powers are  componentwise linear.
Typical examples to which our criteria can be applied
include the vertex cover ideals of certain finite graphs.
\end{abstract}

\maketitle
\section*{Introduction}
Let $K$ be a field, let $S=K[x_1,\ldots, x_n]$ the polynomial
ring in $n$ variables over $K$, and let $\mm=(x_1,\ldots,x_n)$ be its graded maximal ideal. Let $I\subset S$ be a graded ideal. The ideal $I$ is said to be {\em componentwise linear}, if for all $j$ the  ideal $I_{\langle j\rangle}=(I_j)$, generated by the $j$th component of $I$, has a linear resolution. It is known that $I$ is componentwise linear if $(I_{[j]}$ has a linear resolution for all $j\leq \reg I$,
see \cite{HeHi}.
In particular, ideals with linear resolution are componentwise linear. Typical examples of componentwise linear ideals are stable and squarefree stable ideals.

Naively one would expect that for an ideal with linear resolution, all its powers have a linear resolution as well. But this is not the case. A first counterexample was given by Terai who observed that,  if the characteristic of $K$ is zero, then the Stanley--Reisner ideal of the natural triangulation of the real projective plane has a linear resolution, but the square does not. For ideals with linear resolution there exist criteria, among them the so-called $x$-condition
\cite{HHZ},
that allow to test whether all of its powers have a linear resolution as well.  For example it  was shown
\cite{HHZ},
by using the $x$-condition,   that all powers of a  monomial ideal with linear resolution generated in degree 2  have linear resolutions. But this condition  fails in many other cases, for example for the ideal of $2$-minors of a generic symmetric $3\times 3$-matrix which as we shall see has the property that all of its powers have a linear  resolution. To prove this we shall use the following criterion (Corollary~\ref{obvious}) which is one of the main results of the paper. Let $K$ be an infinite field and $I\subset S$ a graded ideal. Then all powers of $I$  are componentwise linear if and only if a generic sequence of linear forms generating the $K$-vector space $S_1$ is a $d$-sequence with respect to the Rees ring $R(I)$ of $I$. Proposition \ref{conclusion} provides a Gr\"obner basis condition  on the defining ideal $R(I)$ that guarantees that a given $K$-basis of $S_1$ is a $d$-sequence with respect to $R(I)$.  In \cite[Proposition 3.1]{Ro} R\"omer presents a  result related to our main theorem, characterizing standard bigraded $K$-algebras with $x$-regularity $0$.

Our criterion can be easily checked by any computer algebra system for graded ideals whose number of generators is not too big. Interesting  examples  are  the ideals defining rational normal scrolls.  In all cases we checked it turned out that all powers of these ideals  have a linear resolution. Due to these computations and due to other known cases \cite{ABW} and \cite{CHV}, one may expect  that all powers of ideals defining rational normal scrolls have a linear resolution.  One even may expect this property holds true for any graded ideal $I$ for which  $S/I$ is Cohen-Macaulay with  minimal multiplicity. However Conca \cite{C} gave an example of a graded  ideal $I$ whose residue class ring is Cohen--Macaulay with minimal multiplicity, but for which $I^2$ does not have a linear resolution.  The best one could hope is that any {\em reduced} graded ideal $I\subset S$ for which $S/I$ is Cohen--Macaulay with  minimal multiplicity has the property that all its powers have a linear resolution.

In the second section of this paper we consider vertex cover ideals of chordal graphs. The vertex cover ideal of a graph $G$ is the Alexander dual of the edge ideal of $G$. In \cite{HHZ1} it was shown that the edge ideal of a chordal graph is Cohen--Macaulay if and only if it is unmixed. This result indicated that edge  ideals of chordal graphs might be sequentially Cohen--Macaulay, and this was indeed shown by Francisco and Van Tuyl \cite[Theorem 3.2]{FT}. According to  \cite[Theorem 2.1]{HeHi}, this result is equivalent to the statement that the vertex cover ideal of a chordal graph is componentwise linear. In this paper we show  that all powers of the vertex cover ideal  of a  star graph  are componentwise linear (Theorem \ref{stargraph}), and that those of a Cohen--Macaulay chordal graph have linear resolutions (Theorem \ref{oberwolfach}). Star graphs are a special class of chordal graphs. By these results and computational evidence we are lead  to conjecture that all powers of the vertex cover ideal of a chordal graph are componentwise linear.

\section{The criterion}

The criterion we are going to prove is based on results on approximation complexes in \cite{HSV} and on a result of R\"omer \cite{Roe} and Yanagawa \cite[Proposition 4.9]{Ya}, which they proved independently. This result  appeared first in  R\"omer's dissertation. A new proof of their  result in a more general frame is given in \cite[Theorem 5.6]{IR}.

Let $M$ be a finitely generated graded  $S$-module with graded minimal free resolution $(G,\varphi)$. Replacing all entries which are not linear by zero in the matrices describing the differentials $\varphi_i$ in $G$, one obtains a complex, which is called the {\em linear part} of $G$.  The linear part of $G$  can be described more naturally as follows: we define a filtration $\mathcal F$ on $G$ by setting
${\mathcal F}_jG_i=\mm^{j-i}G_i$ for all $i$ and $j$, and denote by $\gr_\mm(G)$ the associated graded complex. This complex is isomorphic to  the linear part of $G$.

The module $M$ is called {\em componentwise linear}, if for all $j$ the submodules generated by $M_j$ have a linear resolution.

\begin{Theorem}[R\"omer, Yanagawa]
\label{lpd}
Let $M$ be a finitely generated graded $S$-module with minimal graded free resolution $G$. The following conditions are equivalent:
\begin{enumerate}
\item[{\em (a)}] $M$ is componentwise linear.
\item[{\em (b)}] $\gr_\mm(G)$ is acyclic.
\end{enumerate}
If the equivalent conditions hold, $H_0(\gr_\mm(G))\iso \gr_\mm(M)$.
\end{Theorem}

The link to approximation complexes is given by the next theorem. As usual, the  $\mathcal M$-complex of $M$ with respect to $\mm$ is denoted by ${\mathcal M}(\mm;M)$. It is a linear complex of free $S$-modules. In other words, all entries of the matrices describing the differentials of this complex are linear forms. We refer the reader to \cite{HSV1} for a detailed description of approximation complexes. The following result (\cite[Theorem 5.1]{HSV}) provides another interpretation of the linear part of a resolution.

\begin{Theorem}[Herzog, Simis, Vasconcelos]
\label{hsv1}
Let $M$ be a finitely generated graded $S$-module with minimal graded free resolution $G$. Then
\[
\gr_\mm(G)\iso {\mathcal M}(\mm;M).
\]
\end{Theorem}

In view of Theorem~\ref{lpd} and Theorem~\ref{hsv1} we would like to know when the $\mathcal M$-complex is acyclic. An answer to this question is given by the next result \cite[Theorem 4.1]{HSV}.

\begin{Theorem}[Herzog, Simis, Vasconcelos]
\label{hsv2}
Let $M$ be a finitely generated graded $S$-module and $I$  a proper,  graded ideal. Then the following conditions are equivalent:
\begin{enumerate}
\item[{\em (a)}] ${\mathcal M}(I;M)$ is acyclic.
\item[{\em (b)}] $I$ is generated by a $d$-sequence with respect to $M$.
\end{enumerate}
\end{Theorem}

The concept of $d$-sequences was introduced by Huneke \cite{Hu}. Recall that a sequence ${\bf z}=z_1,\ldots,z_m$ of elements of $S$ with $I=({\bf z})$ is called a $d$-sequence, if
\[
(z_1,\ldots, z_{i-1})M\:_M z_i\sect IM= (z_1,\ldots,z_{i-1})M\quad \text{for}\quad i=1,\ldots, m.
\]

After these preparations we are ready to state and prove  our main result.

\begin{Theorem}
\label{main}
Let $K$ be an infinite field,  and let $I\subset S=K[x_1,\ldots,x_n]$ be a graded ideal. If there exists a $K$-basis  ${\bf z}=z_1,\ldots,z_n$ of $S_1$ such that ${\bf z}$ is a $d$-sequence with respect to the Rees algebra $R(I)$, then all powers of $I$ are componentwise linear.

Conversely,  let ${\bf z}$ be a generic $K$-basis of $S_1$. If all powers of $I$ are componentwise linear, then ${\bf z}$ is a $d$-sequence with respect to $R(I)$.
\end{Theorem}

\begin{proof}
Let ${\bf z}=z_1,\ldots,z_n$ be a $K$-basis of $S_1$ which  is a $d$-sequence with respect to the Rees algebra $R(I)$. Then
\begin{eqnarray}
\label{dsequence}
(z_1,\ldots,z_{i-1})R(I)\: z_i\sect \mm R(I)= (z_1,\ldots,z_{i-1})R(I).
\end{eqnarray}

The Rees ring $R(I)$ is naturally $\ZZ$-graded with $R(I)_k=I^kt^k$ for all ´$k$. Since the elements $z_i$ in this grading are of degree $0$,  equality (\ref{dsequence}) yields the  equalities
\[
(z_1,\ldots,z_{i-1})I^k\: z_i\sect \mm I^k= (z_1,\ldots,z_{i-1})I^k,\quad k=0,1,2,\ldots.
\]
In other words, ${\bf z}$ is a $d$-sequence with respect to $I^k$ for all $k$. Therefore Theorem~\ref{hsv2} implies that the approximation complexes $\mathcal{M}(\mm; I^k)$ are acyclic. Let $G^{(k)}$ be the graded minimal free $S$-resolution of $I^k$. It follows from Theorem~\ref{hsv1} that the associated graded complex $\gr_\mm(G^{(k)})$ is acyclic for all $k$. According Theorem~\ref{lpd}   this implies that $I^k$  is componentwise linear.

Conversely assume that all powers of $I$ are componentwise linear. Then Theorem~\ref{lpd}  and Theorem~\ref{hsv1}   imply that all the approximation complexes  $\mathcal{M}(\mm; I^k)$ are acyclic. Hence  Theorem~\ref{hsv2} asserts that for each $k$ there exists  a $K$-basis of $S_1$ which is a $d$-sequence  $d$-sequence with respect to $I^k$. The proof of \cite[Theorem 4.1]{HSV} shows how the $d$-sequence is constructed: choose any $K$-basis  ${\bf z}=z_1,\ldots,z_n$ of $S_1$ with the property that for all $i>0$ and all $j\in\{1,\ldots,n\}$ we have $H_i(z_1,\ldots,z_j;\gr_\mm(I^k))_\ell=0$ for all $\ell\gg 0$. Then under the assumption that $\mathcal{M}(\mm; I^k)$ is acyclic, this sequence is a $d$-sequence with respect to $I^k$.

Suppose we can choose a $K$-basis  $\bf z$ of $S_1$ such that
\begin{eqnarray}
\label{one}
((z_1,\ldots,z_j)\gr_\mm(I^k):z_{j+1})/(z_1,\ldots,z_j)\gr_\mm(I^k)
 \end{eqnarray}
has finite length for all $k$ and all $j$, then, since $\reg(\gr_\mm(I^k))=0$, \cite[Proposition 20.20]{Ei} implies that
\begin{eqnarray}
\label{all}
(((z_1,\ldots,z_j)\gr_\mm(I^k):z_{j+1})/(z_1,\ldots,z_j)\gr_\mm(I^k))_\ell=0 \quad \text{for all} \quad \ell>0.
\end{eqnarray}
Observe that the regularity of $\gr_\mm(I^k)$ is indeed zero, because  $\gr_\mm(I^k)$ is generated in degree $0$ and  the approximation complex $\mathcal{M}(\mm; I^k)$ provides a linear resolution of $\gr_\mm(I^k)$.

For  a sequence satisfying (\ref{all}) for all $k$ it follows that $H_i(z_1,\ldots,z_j;\gr_\mm(I^k))_\ell=0$ for all $i>0$, all $k$  and $\ell>0$, and so  $\bf z$ is a $d$-sequence with respect to all $I^k$, as explained before. This is equivalent to say that $\bf z$ is a $d$-sequence with respect to $R(I)$. But how can we find  a sequence $\bf z$ such that property (\ref{one}) (and consequently property (\ref{all})) is satisfied for all powers of $I$?

Suppose we have already chosen $z_1,\ldots,z_j$ satisfying (\ref{one}). We apply  the graded version of \cite[Proposition 2.5]{Br}   to the standard graded $S/(z_1,\ldots,z_j)$-algebra $$A=\Dirsum_{k\geq 0}\gr_\mm(I^k)(z_1,\ldots,z_j)\gr_\mm(I^k)$$ to conclude that
\[
\mathcal{A}=\Union_{k\geq 0}\Ass(\gr_\mm(I^k)(z_1,\ldots,z_j)\gr_\mm(I^k))
\]
is a finite set. Hence since $K$ is infinite we may choose a linear form  $z_{j+1}\in S/(z_1,\ldots,z_j)$ which is contained in no prime ideal of $\mathcal{A}$ different from the graded maximal ideal of  $S/(z_1,\ldots,z_j)$. This is the desired next linear form in our sequence.
Hence in each step of the constructing of $\bf z$ we have an open choice. In other words, generic sequences will be $d$-sequences with respect to $R(I)$.
\end{proof}

\begin{Corollary}
\label{obvious}
Let $K$ be an infinite field,  $I\subset S$  a graded ideal and  ${\bf z}$ a generic $K$-basis of $S_1$. The following conditions are equivalent:
\begin{enumerate}
\item[{\em (a)}] All powers of $I$ are componentwise linear;
\item[{\em (b)}] $\bf z$ is a $d$-sequence with respect to $R(I)$.
\end{enumerate}
\end{Corollary}

Let us analyze what it means that a sequence $z_1,\ldots, z_n$ is a $d$-sequence with respect to $R(I)$.  After a change of coordinates we may assume that our given sequence is the sequence $x_1,\ldots,x_n$. Let $f_1,\ldots,f_m$ a homogeneous system of generators of $I$, and let $T=S[y_1,\ldots,y_m]$ be the polynomial ring over $S$ in the variables $y_1,\ldots,y_m$.  We consider the surjective $S$-algebra  homomorphism $T\to R(I)$ with  $y_j\mapsto f_j$ for $j=1,\ldots,m$. Then $R(I)\iso T/J$ where $J$ is the kernel of the map $T\to R(I)$. Identifying $R(I)$ with $T/J$, we see that $x_1,\ldots,x_n$ is a $d$-sequence with respect to $R(I)$ if and only if
\begin{eqnarray}
\label{dsequence}
\qquad (x_1,\ldots,x_{i-1})+J\: x_i\sect ((x_1,\ldots,x_n)+J)=(x_1,\ldots,x_{i-1})+J \quad \text{for}\quad i=1,\ldots,n.
\end{eqnarray}

This is a condition which can be easily checked by any computer algebra system. For example, if we let  $I$ be the ideal of $2$-minors of the generic symmetric matrix
\[
\left(
\begin{array}{ccc} a & b &  c\\
b & d& e\\
c & e& f
\end{array}
\right ),
\]
one easily finds with CoCoA that $a,b,e,d+f,c,f$ is a $d$-sequence on $R(I)$, and hence all powers of $I$ have a linear resolution.

\medskip
Passing to the initial ideal of $J$ with respect to some term order we can deduce a sufficient condition  for the property that all powers of an ideal are componentwise linear. We denote by $J_i$ the ideal which is the image of $J$ under the canonical epimorphism $T\to T/(x_1,\ldots,x_i)$. Then (\ref{dsequence}) implies that $x_1,\ldots,x_n$ is a $d$-sequence on $R(I)$, if and only if
\begin{eqnarray}
\label{dsequence1}
J_{i-1}\: x_i\sect (x_i,\ldots,x_n)+J_{i-1}=J_{i-1} \quad \text{for}\quad i=1,\ldots,n.
\end{eqnarray}

Let $<$ be a monomial order. We denote by $\ini(I)$ the initial ideal with respect to this order.

\begin{Lemma}
\label{initial}
Suppose  there exists a monomial order $<$ on $T/(x_1,\ldots,x_{i-1})$ such that
\begin{eqnarray}
\label{ini}
\ini(J_{i-1})\: x_i\sect ((x_i,\ldots,x_n)+\ini(J_{i-1}))=\ini(J_{i-1}).
\end{eqnarray}
Then (\ref{dsequence1}) holds for this integer $i$.
\end{Lemma}

\begin{proof}
Let $f\in J_{i-1}\: x_i\sect ((x_i,\ldots,x_n)+J_{i-1})$. Then $x_if\in J_{i-1}$, and so $x_i\ini(f)\in \ini(J_{i-1})$. Therefore, $\ini(f)\in \ini(J_{i-1})\: x_i$. On the other hand, since $f\in (x_i,\ldots,x_n)+J_{i-1}$, it follows that $\ini(f)\in\ini((x_1,\ldots,x_i)+J_{i-1})\subset
(x_1,\ldots,x_i)+\ini(J_{i-1})$. Thus our hypothesis implies that $\ini(f)\in \ini(J_{i-1})$. Hence there exists $g\in J_{i-1}$ such that $\ini(g)=\ini(f)$. We may assume that the leading coefficients of $f$ and $g$ are 1, and set $h=f-g$. Then $x_ih=x_if-x_ig\in J_{j-1}$ and $h\in (x_i,\ldots,x_n)+J_{i-1}$. Since $\ini(h)<\ini(f)$, we may assume by induction that $h\in J_{i-1}$. This then implies that $f\in J_{i-1}$, as desired.
\end{proof}

For a monomial ideal $L$ we denote as usual by $G(L)$ the unique minimal set of monomial generators of $L$. We also set $y^b=y_1^{b_1}\cdots y_m^{b_m}$ for any integer vector $b=(b_1,\ldots,b_m)$. Now  the preceding considerations yield

\begin{Proposition}
\label{conclusion}
The sequence $x_1,\ldots,x_n$ is $d$-sequence on $R(I)$, and hence all powers of $I$ are componentwise linear, if for each $i$ there exists monomial order on $T/(x_1,\ldots, x_{i-1})$  such that whenever $x_i$ divides  $u\in G(\ini(J_{i-1}))$, then $u$  the form $u=x_iy^b$  and $x_jy^b\in  \ini(J_{i-1})$ for all $j\geq i$.
\end{Proposition}

In order to prove this proposition we just have to see that  equation (\ref{ini}) holds for all $i$.  This will follow from

\begin{Lemma}
Let $I\subset T$ be a monomial ideal. Then the following conditions are equivalent:
\begin{enumerate}
\item[{\em (a)}] $I:x_1\sect (\mm T +I)=I$.
\item[{\em (b)}] If $x_1$ divides  $u\in G(I)$, then  $u$ is of the form $u=x_1y^b$  and $x_jy^b\in  I$ for all $j=1,\ldots,n$.
\end{enumerate}
\end{Lemma}

\begin{proof}

(a)\implies (b): Let $u\in G(I)$ such that $x_1$  divides $u$. Say, $u=x_1^au'$ with $a\geq 1$ and $x_1$ does not divide $u'$.  Suppose $a>1$; then $x_1^{a-1}u'\not\in I$. But $x_1^{a-1}u' \in I:x_1\sect (\mm T+I)$, a contradiction. Hence we see that $u=x_1u'$. Suppose some $x_i$ divides $u'$ for some $i>1$. Then $u'\in  I:x_1\sect (\mm T+I)$, but $u'\not\in I$, again a contradiction. Thus $u=x_1y^b$ for some exponent vector $b$. The monomial $x_iy^b$ belongs to $I:x_1\sect (\mm T+I)$ for all $i$, and hence $x_iy^b\in I$ for all $i$.

(b)\implies (a): We only need to prove the inclusion $I:x_1\sect (\mm T +I)\subset I$. Our assumptions imply that $G(I)$ can be written as $\{x_1y^{b_1},\ldots, x_1y^{b_k}, u_{k+1},\ldots,u_m\}$ with some integer vectors $b_i$,  and with monomials $u_j\in G(I)$ which are not divisible by  $x_1$ for $j=k+1,\ldots,m$. It follows that
\[
I:x_1=(y^{b_1},\ldots, y^{b_k},u_{k+1},\ldots,u_m).
\]
Thus if $u\in I:x_1\sect \mm T+I$, then $u$ is either a multiple of one of the $u_j$, in which case $u\in I$, or $u=vy^{b_i}$ for some $i$, and $x_j$ divides $v$ for some $j$. Write $v=x_jv'$; then $u=v'x_jy^{b_i}\in I$, according to the assumptions made in (b).
\end{proof}

We demonstrate the use of  Proposition~\ref{conclusion}  with
a simple example.

\begin{Example}
\label{simpleminor}
{\em
Let $I$ be the ideal of $2$-minors of
\[
\left(
\begin{array}{cccc} x_1 & a &  b\\
x_2 & b & c
\end{array}
\right ).
\]
We have $I=(-x_2a + x_1b, -x_2b + x_1c, -b^2 + ac)$ and $J=(-by_1 + ay_2 - xy_3, cy_1 - by_2 + x_2y_3)$.

We want to show that $x_1,x_2,a,c,b$ is a $d$-sequence. We apply Proposition \ref{conclusion} by using the reverse lexicographic order  induced by $x_1>x_2>a>b>c>y_1>y_2>y_3$. Then
\begin{eqnarray*}
 \ini(J_0)&=&(cy_1, by_1, b^2y_2),\quad \ini(J_1)=(cy_1, by_1, b^2y_2), \quad \ini(J_2)=(cy_1, by_1, b^2y_2),\\
 \ini(J_3)&=& (cy_1, by_1, b^2y_2)\quad\text{and}\quad   \ini(J_4)=(by_1, by_2).
\end{eqnarray*}

Thus  we see that the conditions of Proposition \ref{conclusion} are satisfied, and hence $x_1,x_2,a,c,b$ is indeed a $d$-sequence on $R(I)$.
}
\end{Example}

Example \ref{simpleminor} naturally leads us to pose  the following question.
Let $I$ be the ideal of $2$-minors of
\[
\left(
\begin{array}{ccccccc} x & a_1 & a_2 & \cdots & a_{n-1} & a_{n} \\
y & a_2 & a_3 & \cdots & a_{n} & a_{n+1}
\end{array}
\right ).
\]
This ideal defines a rational scroll. Is it true that  the sequence
$x, y, a_1, a_{n+1}, a_{n}, \ldots, a_{2}$
is a $d$-sequence on $R(I)$. This is the case for $n\leq 6$.

\section{Chordal graphs}
Let $[n] = \{1, 2, \ldots, n \}$ denote the vertex set
and $G$ a finite graph on $[n]$ with no loop and no multiple edge.
We write $E(G)$ for the set of edges of $G$.
The {\em edge ideal} of $G$ is the ideal $I(G)$ of
$S = K[x_1, \ldots, x_n]$
which is generated by those monomials $x_ix_j$
with $\{i, j\} \in E(G)$.
A subset $C$ of $[n]$ is called {\em vertex cover} of $G$ if,
for each edge $\{ i, j \}$ of $G$, one has either $i \in C$
or $j \in C$.
A {\em minimal} vertex cover of $G$ is
a vertex cover $C$ with the property that any proper subset
of $C$ cannot be a vertex cover of $G$.
We say that $G$ is {\em unmixed}
if all minimal vertex covers have the same cardinality.
A {\em mixed} graph is a graph which is not unmixed.
The {\em vertex cover ideal} of $G$ is the ideal of $S$
which is generated by those squarefree monomials
$\prod_{i \in C} x_i$ such that
$C$ is a minimal vertex cover of $G$.

The vertex cover ideal of $G$ is the {\em Alexander dual}
of $I(G)$; in other words, the ideal
\[
I_G:=I(G)^\vee= \bigcap_{\{i, j\} \in E(G)} (x_i, x_j).
\]

Recall that a finite graph $G$ is called {\em chordal}
if each cycle of $G$ of length $> 3$ possesses a chord.
In the paper \cite[Theorem 3.2]{FT} by Francisco and Van Tuyl
it is shown that every vertex cover ideal
of a chordal graph is componentwise linear.

\begin{Example}
\label{line}
{\em
(a)
Let $G$ be the line of length $3$.
Thus 
$E(G) = \{
\{1,2\},
\{2,3\},
\{3,4\}
\}$.
Then $G$ is unmixed and
$I_{G} =
(x_1x_3, x_2x_3, x_2x_4)$.
Since
\[
x_1,x_3,x_2+x_4,x_2
\]
is a $d$-sequence with respect to $R(I_G)$,
all powers of $I_G$ are (componentwise) linear.

(b)
Let $G$ be the line of length $4$.
Thus $E(G) = \{
\{1,2\},
\{2,3\},
\{3,4\},
\{4,5\}
\}$.
Then $G$ is mixed and
$I_G
= (x_1x_3x_4,
x_1x_3x_5,
x_2x_4,
x_2x_3x_5)$.
Since
\[
x_1, x_3, x_5,x_2+x_4,x_2
\]
is a $d$-sequence with respect to $R(I_G)$,
all powers of $I_G$
are componentwise linear.
}
\end{Example}

A complete graph on $[n]$ is a finite graph such that $\{ i, j \} \in E(G)$ for all $1 \leq i < j \leq n$.
Let $G_n$ denote the complete graph on $[n]$.
Its vertex cover ideal $I_{G_n}$
is generated by all squarefree monomials of $S$
of degree $n - 1$;
in other words, $I_{G_n}$ is
an ideal of Veronese type (\cite{HH}).
In particular, $I_{G_n}$ is a polymatroidal ideal
(\cite{CH}).  Hence 
all powers of $I_{G_n}$ have linear resolutions.
Theorem \ref{main} then guarantees that a generic
$K$-basis
$\zb = z_1, \ldots, z_n$ of $S_1$ is
a $d$-sequence with respect to $R(I_{G_n})$.

Let $m \geq 1$ and $\Gamma$ a connected graph
on $[n+m]$ such that the induced subgraph of
$G$ on $[n]$ is $G_n$ and that
$\{ i, j \} \not\in E(G)$ if $n < i < j \leq n + m$.
Such a graph is called a {\em star graph}
based on $G_n$.  A star graph based on $G_n$
is a chordal graph.

\begin{Example}
\label{triangle}
{\em
Let $n = 3$ and $m = 3$.
Let $G$ be the
star graph based on $G_3$ with the edges
\[
\{ 1, 4 \},
\{ 2, 4 \},
\{ 2, 5 \},
\{ 3, 5 \},
\{ 1, 6 \},
\{ 3, 6 \}
\]
together with all edges of $G_3$.
Since both $\{ 1, 2, 3 \}$ and
$\{ 2, 3, 4, 6 \}$ are minimal vertex covers of $G$,
it follows that $G$ is mixed. Let $K$ be a field of characteristic $0$.
Since the $K$-basis
\[
x_4, x_5, x_6, x_1+x_2+x_3, 2x_1+3x_2+5x_3, 7x_1+11x_2+13x_3
\]
of $S_1$
is a $d$-sequence with respect to $R(I_{G})$,
all powers of $I_{G}$ are componentwise linear.
}
\end{Example}

The simple computational observation done in Example \ref{triangle}
now grows the following

\begin{Theorem}
\label{stargraph}
All powers of the vertex cover ideal of a star graph
based on $G_n$ are componentwise linear.
\end{Theorem}

\begin{proof}
Let $G$ be a star graph on the  vertex set $\{ x_1, \ldots, x_n, t_1, \ldots, t_m \}$ based on $G_n$ with vertex set
$\{x_1,\ldots,x_n\}$.

Let $T=S[t_1, \ldots, t_m]$
denote the polynomial ring in $n+m$ variables
over $K$ and $I_G \subset T$
the vertex cover ideal of $G$.

First, we observe that $I_G$ is generated by the following monomials:
\[
u_i=\prod_{j=1\atop j\neq i}^nx_j\prod_{k \atop\{x_i,t_k\}\in E(G)}t_k, \quad i=1,\ldots,n,
\]
and
\[
u_{n+1}=\prod_{j=1}^nx_j.
\]
It is clear that all these monomials belong to $I_G$. Conversely, suppose that $C$ is an arbitrary vertex cover of $G$ and $u_C$ its corresponding monomial. If $\{x_1,\ldots,x_n\}\subset C$, then $u_{n+1}$ divides $u_C$. Thus assume
that $x_i\not\in C$ for some $i$. Then $x_j\in C$ for all $j\neq i$, because if $x_j\not\in C$ for some $j\neq i$, then the edge $\{x_i,x_j\}$ is not covered by $C$. It follows that $\{x_1,\ldots, x_{i-1}, x_{i+1},\ldots,x_n\}\subset C$. Moreover since $x_i\not\in C$ we must have $t_k\in C$ for all $k$ with $\{x_i,t_k\}\in E(G)$. This shows that $u_i$ divides $u_C$.

Notice that $u_{n+1}$ is not always part of the minimal monomial set $G(I_G)$ of generators of $I_G$. In fact, $u_{n+1}\in G(I_G)$ if and only if for each $i=1,\ldots,n$ there exists $k$ such that $\{x_i,y_k\}\in E(G)$. For the following considerations, however,
it is not important whether $u_{n+1}$ belongs to $G(I_G)$ or not.

Second, we write the
Rees algebra of $I_G$ as the factor ring
$R(I_G)=T[y_1,\ldots,y_{n+1}]/J$,
where $J$ is the kernel of the $T$-algebra homomorphism given by $y_j\mapsto u_j$ for $j=1,\ldots,n+1$.

We claim that $J$ is generated by the binomials
\[
f_i=x_iy_i-(\prod_{k}t_k)y_{n+1}, \quad i=1,\ldots n,
\]
where for each $i$ the product $\prod_{k}t_k$ in $f_i$ is taken over all $k$
with $\{x_i,t_k\} \in E(G)$.

To see why this is true, we write $L$ for the ideal generated by the binomials $f_1,\ldots,f_n$, and have to show that $L=J$. We first observe that $f_1,\ldots,f_n$ is a regular sequence. Indeed, let $<$ be a lexicographic monomial order with  $y_i>y_{n+1}>x_j,t_k$ for all $i,j, k$ and $i\neq n+1$. Then $\ini_<(f_i)=x_iy_i$ for $i=1,\ldots,n$. Since $\gcd(\ini_<(f_i),\ini_<(f_j)) = 1$ for all $i\neq j$, it follows that $f_1,\ldots,f_n$ is a Gr\"obner basis of $L$, and that $f_1,\ldots,f_n$ is a regular sequence.

Now it follows that
$\dim T[y_1,\ldots,y_{n+1}]/L=\dim T+1=\dim R(I_G)$.
Hence, if $L$ is a prime ideal, then one has $L=J$.
To show that $L$ is a prime ideal,
we consider the ideal $(L,x_1y_1)=( vy_{n+1}, x_1y_1, f_2,\ldots,f_n)$, where $v=\prod_{k}t_k$ and where the product is taken over all $k$ with $\{x_1,t_k\}\in E(G)$. The initial terms of the generators of
the ideal
$(L,x_1y_1)$ are $vy_{n+1}, x_1y_1,\ldots,x_ny_n$. Since these initial terms form a regular sequence of monomials,
it follows as above that $f_1,\ldots,f_n, x_1y_1$ is a regular sequence. In particular
the variable
$x_1$ is a nonzero divisor on  $T[y_1,\ldots,y_{n+1}]/L$. By the same reason all $x_i$ are nonzero divisors on $T[y_1,\ldots,y_{n+1}]/L$. Let $x=\prod_{i=1}^n x_i$; then the natural homomorphism
$T[y_1,\ldots,y_{n+1}]/L\to (T[y_1,\ldots,y_{n+1}]/L)_x$ is a monomorphism. In the localized polynomial ring $T[y_1,\ldots,y_{n+1}]_x$ the ideal $L_x$ is generated by the elements $x_i^{-1}f_i=y_i-x_i^{-1}(\prod_{k}t_k)y_{n+1}$. Thus we see that $(T[y_1,\ldots,y_{n+1}]/L)_x\iso T[y_{n+1}]_x$, which is a domain. It follows that $T[y_1,\ldots,y_{n+1}]/L$ is a domain. In other words, one has $L=J$, as desired.

Now we have that $\ini_<(J)=(x_1y_1,\ldots,x_ny_n)$. Hence  $t_1,\ldots, t_m$ is a regular sequence on $T[y_1,\ldots,y_{n+1}]/\ini_<(J)$, and consequently is a regular sequence on the Rees algebra $R(I_{G})=T[y_1,\ldots,y_{n+1}]/J$. Since any regular sequence is a $d$-sequence, it remains to be shown that we can find a basis $z_1,\ldots ,z_n$ of $S_1$ which is a $d$-sequence on $\overline{R(I_{G})}=R(I_{G})/(t_1,\ldots,t_m)$. Note that
\[
\overline{R(I_{G})}\iso S[y_1,\ldots,y_n]/(x_1y_1,\ldots,x_ny_n).
\]
Thus the theorem follows form the following Lemma \ref{special}.
\end{proof}

\begin{Lemma}
\label{special}
Let $R=S[y_1,\ldots,y_n]/(x_1y_1,\ldots,x_ny_n)$, and $z_1,\ldots,z_n$   a generic $K$-basis of $S_1$. Then $z_1,\ldots,z_n$ is a $d$-sequence on $R$.
\end{Lemma}

\begin{proof}
We will apply Theorem~\ref{hsv2} and have to show that the approximation complex $\mathcal{M}(\mm;R)$ is acyclic, where $\mm=(x_1,\ldots,x_n)$. Recall that
\[
\mathcal{M}(\mm;R)_i=H_i(\mm;R)\tensor A(-i),
\]
where $A=R[e_1,\ldots,e_n]$ is a polynomial ring over $R$ in the variables $e_1,\ldots,e_n$, and where $H_i(\mm;R)$ is the $i$th Koszul homology of the sequence $x_1,\ldots,x_n$ with values in $R$. Since $R$ is a complete intersection, it follows from
\cite[Theorem 2.3.9]{BH}
that $H(\mm;R)$ is the $i$th exterior algebra of $H_1(\mm;R)$. Since $H_1(\mm;R)$ is a free $K[y_1,\ldots,y_n]$-modules whose basis is given by the homology classes $[y_ie_i]$ of the cycles $y_ie_i$, we see that $H_i(\mm;R)$ is free $(R/\mm)[y_1,\ldots,y_n]$-module with basis $\{b_J \: J=\{j_1,j_2,\ldots, j_i\}\; 1\leq j_1<j_2<\ldots <j_i\leq n\}$, where
\[
b_J=[y_{j_1}y_{j_2}\cdots y_{j_i}e_{j_1}\wedge e_{j_2}\wedge \cdots  \wedge e_{j_i}].
\]
Thus $\mathcal{M}(\mm;R)_i$ can be identified with the free $B$-module $\Dirsum_{J, |J|=i}b_JB$, where $$B=K[y_1,\ldots,y_n,e_1,\ldots,e_n].$$
After this identification the differential of $\mathcal{M}(\mm;R)$ is given by
\[
\partial(b_J)=\sum_{k=1}^i(-1)^{k+1}y_{j_k}e_{j_k}b_{J\setminus \{j_k\}}\quad \text{for}\quad J=\{j_1<j_2<\ldots< j_i\}
\]
Hence $\mathcal{M}(\mm;R)$ is isomorphic to the Koszul complex  $K(y_1e_1,\ldots, y_ne_n;B)$. Since the sequence
$y_1e_1,\ldots, y_ne_n$ is a regular sequence,
it follows then that $\mathcal{M}(\mm;R)$ is acyclic, as desired.
\end{proof}

Theorem~\ref{stargraph} and  computational evidence, including Example \ref{Boston} leads us to
present the following

\begin{Conjecture}
\label{chordalconjecture}
{\em
All powers of the vertex cover ideal of a chordal graph
are componentwise linear.
}
\end{Conjecture}

\begin{Example}
\label{Boston}
{\em
Let $G$ be the chordal graph drawn below.
The graph $G$ is no longer a star graph.
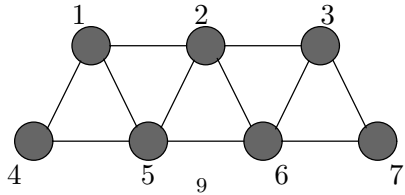
\begin{center}
\unitlength 0.1in
\begin{picture}( 20.4000,  8.8500)(  7.2000,-26.0000)
%
\special{pn 8}%
\special{sh 0.600}%
\special{ar 860 2500 100 100  0.0000000 6.2831853}%
%
\special{pn 8}%
\special{sh 0.600}%
\special{ar 1460 2500 100 100  0.0000000 6.2831853}%
%
\special{pn 8}%
\special{sh 0.600}%
\special{ar 2060 2500 100 100  0.0000000 6.2831853}%
%
\special{pn 8}%
\special{sh 0.600}%
\special{ar 2660 2500 100 100  0.0000000 6.2831853}%
%
\special{pn 8}%
\special{sh 0.600}%
\special{ar 1160 2000 100 100  0.0000000 6.2831853}%
%
\special{pn 8}%
\special{sh 0.600}%
\special{ar 1760 2000 100 100  0.0000000 6.2831853}%
%
\special{pn 8}%
\special{sh 0.600}%
\special{ar 2360 2000 100 100  0.0000000 6.2831853}%
%
\special{pn 8}%
\special{pa 960 2500}%
\special{pa 1360 2500}%
\special{fp}%
\special{pa 1560 2500}%
\special{pa 1960 2500}%
\special{fp}%
\special{pa 2160 2500}%
\special{pa 2560 2500}%
\special{fp}%
\special{pa 2260 2000}%
\special{pa 1860 2000}%
\special{fp}%
\special{pa 1660 2000}%
\special{pa 1260 2000}%
\special{fp}%
%
\special{pn 8}%
\special{pa 930 2436}%
\special{pa 1106 2080}%
\special{fp}%
%
\special{pn 8}%
\special{pa 1526 2430}%
\special{pa 1700 2076}%
\special{fp}%
%
\special{pn 8}%
\special{pa 2126 2436}%
\special{pa 2300 2080}%
\special{fp}%
%
\special{pn 8}%
\special{pa 2600 2430}%
\special{pa 2426 2076}%
\special{fp}%
%
\special{pn 8}%
\special{pa 1996 2420}%
\special{pa 1820 2066}%
\special{fp}%
\put(10.6000,-18.8500){\makebox(0,0)[lb]{$1$}}%
\put(17.0000,-18.8500){\makebox(0,0)[lb]{$2$}}%
\put(23.6000,-18.8500){\makebox(0,0)[lb]{$3$}}%
\put(7.2000,-27.2500){\makebox(0,0)[lb]{$4$}}%
\put(14.2000,-27.2500){\makebox(0,0)[lb]{$5$}}%
\put(27.2000,-27.2500){\makebox(0,0)[lb]{$7$}}%
\put(21.2000,-27.2500){\makebox(0,0)[lb]{$6$}}%
%
\special{pn 8}%
\special{pa 1400 2426}%
\special{pa 1226 2070}%
\special{fp}%
\end{picture}%

\end{center}

\bigskip

\noindent
Then $G$ is mixed and $I_G$ is generated by
$$
x_1 x_3 x_5 x_6 , \  x_2 x_3 x_4 x_5 x_6 , \  x_1 x_2 x_4 x_6 x_7 , \  x_1 x_2 x_3 x_4 x_6 ,
$$
$$
x_2 x_3 x_4 x_5 x_7 , \  x_1 x_2 x_3 x_5 x_7 , \  x_1 x_2 x_5 x_6 x_7 ,  \ x_2 x_4 x_5 x_6 x_7.
$$
We can easily find a $d$-sequence which is a ``generic'' $K$-basis
of $S_1$ created by {\tt CoCoA} with ``{\tt Randomized}.''
Hence all powers of
the vertex cover ideal of $G$ are componentwise linear.
}
\end{Example}

In \cite{HHZ1} the Cohen--Macaulay chordal graphs are classified. Let $G$ be a finite  graph on $[n]$, and $\Delta(G)$ its clique complex, that is to say, the simplicial complex whose faces are the cliques (i.e. complete subgraphs) of $G$. It is shown in \cite{HHZ1} that if  $G$ is chordal, then $G$ is  Cohen--Macaulay if and only if $[n]$ is the disjoint union of those  facets of $\Delta(G)$ which have  a free vertex. A vertex of a facet is {\em free}, if it belongs to no other facet.

In support of Conjecture \ref{chordalconjecture} we have the following result.

\begin{Theorem}
\label{oberwolfach}
Let $G$ be a graph on $[n]$,  and suppose that $[n]$ is the disjoint union of those facets of the clique complex of $G$ with a free vertex. Then all powers of  $I_G$ have a linear resolution.
\end{Theorem}

\begin{proof}
Let $F_1,\ldots,F_m$ be the facets of $\Delta(G)$ which have a free vertex. Since $[n]=F_1\union F_2\union\cdots \union F_s$ is a disjoint union, we may assume that if $i\in F_p$, $j\in F_q$ and $p<q$, then $i<j$. In particular, $1\in F_1$ and $n\in F_s$. Moreover, we may assume that if $i_1,i_2\in F_i$ where $i_1$ is a nonfree vertex and $i_2$ is a free vertex, then $i_1<i_2$.

Observe that any minimal vertex cover of $G$ is of the following form:
\[
(F_1\setminus \{a_1\})\union (F_2\setminus \{a_2\})\union \cdots \union (F_s\setminus \{a_s\}), \quad \text{where} \quad a_j\in F_j.
\]
In particular, $G$ is unmixed and all generators of $I_G$ have degree $n-s$.

Now let $R(I_G)$ be the Rees algebra of vertex cover ideal of $G$. Suppose $u_1,\ldots,u_m$ is the minimal  set of monomial generators of $I_G$. Then there is a surjective $K$-algebra homomorphism
\[
K[x_1,\ldots,x_n,y_1,\ldots,y_m]\To R(I_G)\quad x_i\mapsto x_i \quad\text{and}\quad y_j\mapsto u_j,
\]
whose kernel $J$ is a binomial ideal. Let $<$ be the lexicographic order induced by the ordering $x_1>x_2>\cdots >x_n>y_1>\cdots>y_m$.  We are going to show that the generators of $\ini_>(J)$ are at most of degree $1$ in the $x_i$. This is the so-called $x$-condition and it implies that all powers of $I_G$ have a linear resolution, see \cite{HHZ}. Suppose that $x_{i_1}x_{i_2}\cdots x_{i_p}y_{j_1}y_{j_2}\cdots y_{j_q}$  with $i_1\leq i_2\leq \ldots \leq i_p$ is a minimal generator of $\ini_<(J)$. Then
\begin{eqnarray}
\label{nice}
x_{i_1}x_{i_2}\cdots x_{i_p}y_{j_1}y_{j_2}\cdots y_{j_q}-x_{k_1}x_{k_2}\cdots x_{k_p}y_{\ell_1}y_{\ell_2}\cdots y_{\ell_q}\in J.
\end{eqnarray}
It follows that $i_1<\min\{k_1,\ldots,k_p\}$, and there exists an index $j_r$ such that $x_{i_1}$ does not divide $u_{j_r}$. Say, $i_1\in F_c$. Then (\ref{nice}) implies that there exists $d\in [p]$ with $k_d\in F_c$. In particular, $i_1\neq \max\{i\:\; i\in F_d\}$. Let $i_0=\max\{i\:\; i\in F_d\}$. Since $i_0$ is a free vertex, it follows that $x_{i_1}(u_{j_r}/x_{i_0})$ is a minimal generator of $I_G$, say, $u_g$. Therefore, $f=x_{i_1}y_{j_r}-x_{i_0}y_g\in J$ and $\ini(f)=x_{i_1}y_{j_r}$ divides $x_{i_1}x_{i_2}\cdots x_{i_p}y_{j_1}y_{j_2}\cdots y_{j_q}$, as desired.
\end{proof}

\end{document}